\documentclass[11pt]{article}

\usepackage[margin=1.15in]{geometry}
\usepackage{amsmath,amssymb,amsthm}
\usepackage{microtype}
\usepackage[hidelinks]{hyperref}

\newtheorem{theorem}{Theorem}[section]
\newtheorem{lemma}[theorem]{Lemma}
\newtheorem{proposition}[theorem]{Proposition}
\newtheorem{corollary}[theorem]{Corollary}

\newcommand{\Mex}{\operatorname{mex}}

\title{Unique Winning Opening Move in Three-Row Chomp}
\author{Erez Sheiner}
\date{May 22, 2026}

\begin{document}
\maketitle

\begin{abstract}
Chomp was introduced by Gale in 1974 \cite{Gale1974}. In the same paper,
Gale reported that the $3\times n$ games had been completely analyzed for
$n\le 100$, with a unique winning first move in every case, and asked whether
winning first moves are unique in general. Although the general uniqueness
statement is false \cite[Section~7.1]{BrouwerEtAl2005}, we prove that the
three-row uniqueness phenomenon suggested by Gale's computations holds for
all $n$: every $3\times n$ Chomp rectangle has exactly one winning opening
move. This settles the three-row case of Gale's 52-year-old first-move
uniqueness question.

The proof is carried out in the two-variable recurrence introduced by
Brouwer, Horv\'ath, Moln\'ar-S\'aska, and Szab\'o \cite{BrouwerEtAl2005}
for the function $f(q,r)$ whose values encode the $P$-positions. The main
local ingredient is a rightmost-hole principle: if a value $p$ is absent from
the set $C(q,r)$ but belongs to all corresponding sets $C(t,r)$ for $q<t<p$,
then all intermediate values $q+1,\ldots,p-1$ are forced to belong to
$C(q,r)$. This separates the diagonal values from the starts of constant
rows, and yields a partition of the positive integers into the two possible
types of winning opening moves. It also identifies the row of the unique
opening move: no first-row opening move is winning; the second-row and
third-row cases are precisely the two complementary Chomp sequences A029900
and A029901.
\end{abstract}

\section{Introduction}

Chomp, introduced by Gale \cite{Gale1974}, is played on a finite
$n\times m$ binary matrix $A$. The game starts with $A_{ij}=1$ for all $i,j$,
and a move $(p,q)$ is allowed if $A_{pq}=1$. The move $(p,q)$ sets
$A_{ij}=0$ for every $i\ge p$, $j\ge q$. The player forced to make the move
$(1,1)$ loses. In the three-row case we write a position as
\[
        [p,q,r],\qquad p\ge q\ge r\ge 0,
\]
where $p,q,r$ are the row lengths in nonincreasing order. A $P$-position is a
position from which the previous player wins, and an $N$-position is a
position from which the next player wins. For further historical references,
including the discussion of Chomp in \emph{Winning Ways} \cite[p.~632]{WinningWays2003},
see Brouwer's Chomp page \cite{BrouwerChompPage}.

A move reducing the second row of $[p,q,r]$ to length $a<q$ gives
$[p,a,\min(a,r)]$; in particular, if $a<r$, the resulting position is
$[p,a,a]$. Similarly, reducing the first row to length $a<p$ gives
$[a,\min(a,q),\min(a,r)]$.

The three-row case has been studied through recurrences, computation, and
renormalization methods. Zeilberger gave an early recurrence-based analysis of
three-row Chomp \cite{Zeilberger2001}. Brouwer, Horv\'ath, Moln\'ar-S\'aska,
and Szab\'o introduced the two-variable function used below and developed the
associated table of $P$-positions \cite{BrouwerEtAl2005}. Brouwer's Chomp
page records computations with $q,r\le 130{,}000$ and lists related OEIS
sequences, including the two opening-move sequences A029900 and A029901
\cite{BrouwerChompPage,OEISA029900,OEISA029901}. Friedman and Landsberg
studied the same structure from a nonlinear-dynamics and renormalization
viewpoint \cite{FriedmanLandsberg2007}. A recent preprint of Garg on
$4\times n$ Chomp also places the three-row case in this context and states
that a proof of three-row uniqueness had remained out of reach
\cite{Garg2026}. The present note provides such a proof directly from the
recurrence; equivalently, it proves the complementarity of A029900 and A029901.

\begin{theorem}[Unique opening move]\label{thm:main}
For every positive integer $n$, the initial three-row Chomp rectangle
\[
        [n,n,n]
\]
has exactly one winning opening move.
\end{theorem}

The rest of this note proves the theorem.

\section{The recurrence}

The following elementary lemma explains the meaning of the function used in
the recurrence.

\begin{lemma}[Existence and uniqueness of the encoded $P$-position]
\label{lem:encodedexistence}
For every $q\ge r\ge 0$, there is a unique integer $p>r$ such that one of the
following holds:
\begin{enumerate}
\item[(i)] $r<p<q$ and $[p,p,r]$ is a $P$-position;
\item[(ii)] $p\ge q$ and $[p,q,r]$ is a $P$-position.
\end{enumerate}
In either case, the corresponding $P$-position is obtained from any sufficiently
long position $[k,q,r]$ by reducing the first row to length $p$.
\end{lemma}

\begin{proof}
First observe uniqueness. Suppose $p_1<p_2$ are two integers satisfying the
lemma. If $p_i<q$, the position corresponding to $p_i$ is $[p_i,p_i,r]$; if
$p_i\ge q$, it is $[p_i,q,r]$.

In all cases, the position corresponding to $p_1$ is obtained from the
position corresponding to $p_2$ by reducing the first row to length $p_1$.
This is a legal move from one $P$-position to another, impossible.

It remains to prove existence. We argue by induction on $q+r$. The case
$q=r=0$ is clear: among the one-row positions $[p,0,0]$, the unique
$P$-position is $[1,0,0]$.

Now let $q\ge r\ge 0$, and assume the assertion has already been proved for
all pairs $(q',r')$ with $q'\ge r'$ and $q'+r'<q+r$.

Consider first the values of $p$ for which $[p,q,r]$ can move to a
$P$-position by reducing the second row. If the second row is reduced to
$a<q$, the resulting position is
\[
        [p,a,\min(a,r)].
\]
The lower pair $(a,\min(a,r))$ has smaller sum than $(q,r)$. Hence, by the
induction hypothesis, there is at most one value of $p$ for which this
resulting position is a $P$-position.

Similarly, if the third row is reduced to $b<r$, the resulting position is
\[
        [p,q,b],
\]
and the lower pair $(q,b)$ has smaller sum than $(q,r)$. Hence, again by the
induction hypothesis, there is at most one value of $p$ for which this
resulting position is a $P$-position.

Thus only finitely many values of $p$ are blocked by moves reducing the
second or third row. Choose an integer $p\ge q$ which is not one of these
blocked values.

If reducing the first row of $[p,q,r]$ to some length $k$ with $r<k<p$
gives a $P$-position, then existence is proved, since this $k$ satisfies the
lemma. Otherwise, no move reducing the first row to a length greater than
$r$ leads to a $P$-position. A move reducing the first row to a length
$1\le k\le r$ gives the rectangle $[k,k,k]$, which is an $N$-position by the
standard strategy-stealing argument, and the move to length $0$ is the losing
move that takes the poisoned square.

By the choice of $p$, no move reducing the second or third row of $[p,q,r]$
leads to a $P$-position. Therefore $[p,q,r]$ has no move to a $P$-position,
and hence it is itself a $P$-position. Since $p\ge q$, this is alternative
(ii).
\end{proof}

We denote the unique integer in Lemma~\ref{lem:encodedexistence} by
$f(q,r)$. Thus, if
\[
        r<f(q,r)<q,
\]
then
\[
        [f(q,r),f(q,r),r]
\]
is a $P$-position, while if
\[
        f(q,r)\ge q,
\]
then
\[
        [f(q,r),q,r]
\]
is a $P$-position. Conversely, every three-row $P$-position is encoded in
this way.

The recurrence introduced by Brouwer et al. \cite{BrouwerEtAl2005} is based
on the following blocked values. When computing $f(q,r)$, a value
$p$ is blocked if choosing it would give a position with a move to an already
encoded $P$-position. We separate the blocked values according to whether
the move reduces the second row or the third row.

If the second row is reduced to some length $a<q$, then the lower two row
lengths become
\[
        a,\min(a,r).
\]
The corresponding encoded $P$-position is therefore given by
$f(a,\min(a,r))$. Thus the values blocked by reducing the second row are
\[
        R(q,r)=\{f(a,\min(a,r)):0\le a<q\}.
\]
Equivalently,
\[
        R(q,r)=
        \{f(a,a):0\le a<r\}\cup\{f(a,r):r\le a<q\}.
\]

If the third row is reduced to some length $b<r$, then the lower two row
lengths become $q,b$. Hence the values blocked by reducing the third row are
\[
        C(q,r)=\{f(q,b):0\le b<r\}.
\]
Finally, let
\[
        B(q,r)=R(q,r)\cup C(q,r).
\]
Thus $B(q,r)$ is the set of all values already blocked when computing
$f(q,r)$.

The recurrence of Brouwer et al. is
\begin{equation}\label{eq:recurrence}
 f(q,r)=
 \begin{cases}
   f(q-1,r), & q>r\ \text{and}\ f(q-1,r)<q,\\
   \Mex B(q,r), & \text{otherwise},
 \end{cases}
\end{equation}
where $\Mex$ denotes the minimum excluded value
\[
        \Mex(A)=\min\{k>0:k\notin A\}.
\]
We shall call $(q,r)$ a mex cell when the second line of
\eqref{eq:recurrence} applies.

The first line of the recurrence is needed because $B(q,r)$ records only
the values blocked by moves reducing the second or third row. It does not
record possible moves reducing the first row. If
\[
        k=f(q-1,r)<q,
\]
then $k$ already encodes the $P$-position $[k,k,r]$. Hence any position
$[p,q,r]$ with $p\ge q$ has a legal move to this $P$-position by reducing
the first row to length $k$. In such a case recomputing a mex from $B(q,r)$
alone may give a value whose corresponding position is actually an
$N$-position, because it can still move to $[k,k,r]$. Therefore the
recurrence first checks whether such an old encoded $P$-position is already
reachable by reducing the first row. If so, the value remains
$f(q-1,r)$; only otherwise do we compute the new value by the mex rule.

For orientation, the first few rows of the function have simple closed forms.
The recurrence gives
\[
        f(q,0)=q+1 \qquad(q\ge 0),
\]
while
\[
        f(1,1)=3,\qquad f(q,1)=2\quad(q\ge 2),
\]
and
\[
        f(q,2)=q+2\qquad(q\ge 2).
\]
Thus the recurrence contains both unbounded linear rows and rows that become
constant; the proof below shows that these two behaviours separate the two
possible types of winning opening moves.

\begin{lemma}\label{lem:mexcell}
If $0\le r\le q$ and $f(q,r)>q$, then $(q,r)$ is a mex cell and
\[
        \{1,2,\ldots,q\}\subseteq B(q,r).
\]
\end{lemma}

\begin{proof}
If the first line of \eqref{eq:recurrence} applied, then
$f(q,r)=f(q-1,r)<q$, contrary to $f(q,r)>q$. Hence
$f(q,r)=\Mex B(q,r)$. Since this mex is larger than $q$, every positive
integer at most $q$ lies in $B(q,r)$.
\end{proof}

We shall also use the following elementary consequences of the fact that no
move from a $P$-position can lead to another $P$-position.

\begin{lemma}\label{lem:nonattacking}
The following hold.
\begin{enumerate}
\item[(a)] For fixed $q$, the values $f(q,0),f(q,1),\ldots,f(q,q)$ are distinct.
\item[(b)] If $s<t_1<t_2<p$ and $f(t_1,s)=p$, then $f(t_2,s)\ne p$.
\item[(c)] If $f(p,r)=p$, then $p\notin R(q,r)$ for every $q<p$.
\end{enumerate}
\end{lemma}

\begin{proof}
We use only the defining property that no option of a $P$-position is again a
$P$-position.

First, by Lemma~\ref{lem:encodedexistence}, $f(q,b)>b$ for every $b\le q$.

For (a), suppose $b_1<b_2\le q$ and $f(q,b_1)=f(q,b_2)=p$. Since $p>b_2$,
the two encoded $P$-positions differ only by reducing the third row from
$b_2$ to $b_1$. This would give a legal move from one $P$-position to another.

For (b), since $p>t_2$, the equalities $f(t_1,s)=p$ and $f(t_2,s)=p$ encode
the $P$-positions $[p,t_1,s]$ and $[p,t_2,s]$. The first is obtained from the
second by reducing the second row. Again this would be a legal move from one
$P$-position to another.

For (c), suppose that $f(p,r)=p$. We show that no value contributing to
$R(q,r)$, with $q<p$, can be equal to $p$.

If $r\le a<q<p$ and $f(a,r)=p$, then the $P$-position $[p,a,r]$ is obtained
from the $P$-position $[p,p,r]$ by reducing the second row. This is
impossible.

If $0\le a<r$ and $f(a,a)=p$, then the $P$-position $[p,a,a]$ is obtained
from the $P$-position $[p,p,r]$ by reducing the second row to length $a$.
This is again impossible. Therefore $p\notin R(q,r)$ for every $q<p$.
\end{proof}

\section{A local blocking principle}

The following lemma is the local blocking step used later. Recall that
\[
        C(q,r)=\{f(q,b):0\le b<r\},
\]
so membership in \(C(q,r)\) records values blocked by reducing the third row.

\begin{lemma}[Interval blocking]\label{lem:intervalblocking}
Let $0\le r<q<p$. Suppose
\[
        p\notin C(q,r)
\]
and
\[
        p\in C(t,r)\qquad\text{for every } q<t<p.
\]
Then
\[
        \{q+1,q+2,\ldots,p-1\}\subseteq C(q,r).
\]
\end{lemma}

\begin{proof}
Fix $t$ with $q<t<p$. Since $p\in C(t,r)$, there is a unique $s_t<r$ such
that
\[
        f(t,s_t)=p.
\]
The uniqueness follows from Lemma~\ref{lem:nonattacking}(a). Moreover, if
$t_1<t_2$ and $s_{t_1}=s_{t_2}$, then Lemma~\ref{lem:nonattacking}(b) is
contradicted. Hence the rows $s_t$ are distinct as $t$ varies over
$q+1,\ldots,p-1$.

For each such $t$, put
\[
        u_t=f(q,s_t).
\]
We claim that
\begin{equation}\label{eq:utbounds}
        q<u_t<p.
\end{equation}

Indeed, consider the $P$-position $[p,t,s_t]$. Reducing the second row from
$t$ to $q$ gives the option $[p,q,s_t]$. Since this position is an option of a
$P$-position, it is an $N$-position; because the game is finite, it therefore
has a move to a $P$-position.

That winning reply cannot be obtained by reducing the second row of
$[p,q,s_t]$, since the same resulting position would already be reachable from
$[p,t,s_t]$. It also cannot be obtained by reducing the third row, since that
would give a $P$-position of the form $[p,q,b]$ with $b<s_t<r$, i.e.
$f(q,b)=p$, contradicting $p\notin C(q,r)$.

Therefore the winning reply from $[p,q,s_t]$ is obtained by reducing the first
row. If that reduction made the first row at most $q$, then the same resulting
position would already be reachable directly from $[p,t,s_t]$ by reducing the
first row, impossible. Hence the resulting $P$-position has the form
$[u,q,s_t]$ with $q<u<p$. By the definition of $f$, this $u$ must be
$f(q,s_t)=u_t$. Thus $q<u_t<p$, proving \eqref{eq:utbounds}.

The values $u_t$ are distinct, since the rows $s_t$ are distinct and,
by Lemma~\ref{lem:nonattacking}(a), the values $f(q,b)$ with $0\le b\le q$
are distinct. Hence
$t\mapsto u_t$ is an injection from $\{q+1,\ldots,p-1\}$ into itself. It is
therefore a bijection, and every value $q+1,\ldots,p-1$ occurs as $f(q,s_t)$
for some $s_t<r$. This is exactly the asserted inclusion in $C(q,r)$.
\end{proof}

\begin{lemma}[Rightmost-hole lemma]\label{lem:rightmosthole}
Let $0\le r<q<p$. Suppose
\[
        p\notin C(q,r)
\]
and
\[
        p\in C(t,r)\qquad\text{for every } q<t<p.
\]
If
\[
        f(q,r)>q,
\]
then
\[
        f(q,r)\ge p.
\]
\end{lemma}

\begin{proof}
By Lemma~\ref{lem:intervalblocking}, the values $q+1,q+2,\ldots,p-1$ lie in
$C(q,r)$ and hence in $B(q,r)$. By Lemma~\ref{lem:mexcell}, the hypothesis
$f(q,r)>q$ implies that $1,2,\ldots,q$ also lie in $B(q,r)$, and that
$f(q,r)=\Mex B(q,r)$. Thus every value $1,2,\ldots,p-1$ lies in $B(q,r)$, so
the mex is at least $p$.
\end{proof}

\section{Separating diagonal values and row-starts}

We next prove the following diagonal maximality property.

\begin{lemma}[Maximum at the diagonal for fixed $q$]\label{lem:diagonalmax}
For every $q\ge 0$, the diagonal value $f(q,q)$ is the largest among the
values
\[
        f(q,0),f(q,1),\ldots,f(q,q).
\]
Equivalently,
\[
        f(q,q)=\max_{0\le r\le q} f(q,r).
\]
In particular,
\[
        f(q,q)>q.
\]
\end{lemma}

\begin{proof}
Suppose, to the contrary, that \(f(q,q)\) is not the largest among
\(f(q,0),\ldots,f(q,q)\). Let \(f(q,r)\), with \(r<q\), be the largest
of these values. By Lemma~\ref{lem:nonattacking}(a), the \(q+1\) values
are distinct positive integers, so their maximum is greater than \(q\).
Thus \(f(q,r)>q\). Lemma~\ref{lem:mexcell} therefore implies that
\((q,r)\) is a mex cell and
\[
        f(q,r)=\Mex B(q,r).
\]

For each \(s\) with \(r<s\le q\), we have \(f(q,s)<f(q,r)\). Since
\(f(q,r)=\Mex B(q,r)\), this implies \(f(q,s)\in B(q,r)\). It cannot occur in
\(C(q,r)\), by distinctness for fixed \(q\), Lemma~\ref{lem:nonattacking}(a).

It also cannot occur among the values $f(t,t)$ with $0\le t\le r$. Indeed, if
$f(t,t)=f(q,s)=m$ for some $t\le r<s$, then $m>s>t$. The $P$-position encoded
by $f(q,s)=m$ has second row of length at least $s$, and reducing that second
row to $t$ gives the $P$-position $[m,t,t]$, encoded by $f(t,t)=m$. This is
impossible.

Hence each of the $q-r$ values $f(q,r+1),\ldots,f(q,q)$ must occur at a
position $f(t,r)$ with $r<t<q$. There are only $q-r-1$ such positions, a
contradiction. Thus the diagonal entry is maximal. Since the values
$f(q,0),f(q,1),\ldots,f(q,q)$ are $q+1$ distinct positive integers, their
maximum is at least $q+1$, and therefore
$f(q,q)>q$.
\end{proof}

The following propagation lemma is the main exclusion step.

\begin{lemma}[Row-start propagation]\label{lem:rowstartpropagation}
Assume
\[
        f(p,r)=p.
\]
Then, for every $q$ with
\[
        r<q<p,
\]
one has
\[
        p\in C(q,r).
\]
\end{lemma}

\begin{proof}
Suppose not, and choose the largest $q$ with $r<q<p$ such that
$p\notin C(q,r)$. By maximality,
\[
        p\in C(t,r)\qquad\text{for every }q<t<p.
\]
We first show that
\begin{equation}\label{eq:fqrlarge}
        f(q,r)>q.
\end{equation}
If $f(q,r)<q$, then the recurrence enters the constant branch at the next
step, since $f(q,r)<q+1$, and so $f(q+1,r)=f(q,r)$. Inductively, the value remains constant as the first argument increases up to
$p$, which contradicts $f(p,r)=p$. If
$f(q,r)=q$, then again $f(q,r)<q+1$, so the value becomes constant from first argument $q+1$ onward, now with
value $q$, and again $f(p,r)\ne p$. Hence
\eqref{eq:fqrlarge} holds.

The rightmost-hole lemma now gives
\begin{equation}\label{eq:fqrgep}
        f(q,r)\ge p.
\end{equation}
On the other hand, the value $p$ is absent from $B(q,r)$. It is absent from
$C(q,r)$ by the choice of $q$, and it is absent from $R(q,r)$ by
Lemma~\ref{lem:nonattacking}(c), since $q<p$ and $f(p,r)=p$.

Since \eqref{eq:fqrlarge} makes $(q,r)$ a mex cell, $f(q,r)=\Mex B(q,r)$. A
mex cannot skip over an absent value $p$; together with \eqref{eq:fqrgep},
this forces
\[
        f(q,r)=p.
\]
Since $r<q<p$, this equality encodes the $P$-position $[p,q,r]$. But
$f(p,r)=p$ encodes the $P$-position $[p,p,r]$. The position $[p,q,r]$ is
obtained from $[p,p,r]$ by reducing the second row from $p$ to $q$, a legal
move from one $P$-position to another. This is impossible. Therefore no such
$q$ exists.
\end{proof}

Define
\[
        D=\{f(a,a):a\ge 0\}
\]
and
\[
        S=\{p\ge 1:\text{ there exists } r<p \text{ such that } f(p,r)=p\}.
\]
The elements of $D$ are the first coordinates of diagonal $P$-positions. The
elements of $S$ are the first coordinates at which some row starts to be
constant.

\begin{proposition}\label{prop:disjoint}
The sets $D$ and $S$ are disjoint.
\end{proposition}

\begin{proof}
Assume, for a contradiction, that
\[
        p=f(a,a)
\]
and also
\[
        f(p,r)=p
\]
for some $r<p$.

By Lemma~\ref{lem:diagonalmax}, $p=f(a,a)>a$. If $r\ge a$, then the
$P$-position $[p,a,a]$ is obtained from the $P$-position $[p,p,r]$ by reducing
the second row to $a$, impossible. Hence
\[
        r<a<p.
\]
Apply Lemma~\ref{lem:rowstartpropagation} with $q=a$. It gives
\[
        p\in C(a,r),
\]
so $p=f(a,b)$ for some $b<r<a$. But also $p=f(a,a)$, contradicting distinctness for fixed $a$, Lemma~\ref{lem:nonattacking}(a).
\end{proof}

\begin{proposition}\label{prop:partition}
Every positive integer belongs to exactly one of $D$ and $S$.
\end{proposition}

\begin{proof}
Let $p\ge 1$. The rectangle $[p,p,p]$ is an $N$-position by the standard
strategy-stealing argument. A winning move from it cannot be a move that cuts
all three rows to a smaller rectangle $[m,m,m]$ with $m\ge 1$, since that is
again an $N$-position; the move that takes the poisoned square is losing.
Therefore a winning opening move has one of the following two forms:
\[
        [p,a,a]\qquad (a<p),
\]
or
\[
        [p,p,r]\qquad (r<p).
\]
In the first case $p=f(a,a)$, so $p\in D$. In the second case $f(p,r)=p$, so
$p\in S$. Thus every positive integer belongs to $D\cup S$.

By Proposition~\ref{prop:disjoint}, the two sets are disjoint. Hence every
positive integer belongs to exactly one of them.
\end{proof}

\begin{proof}[Proof of Theorem~\ref{thm:main}]
A move from $[n,n,n]$ is winning if and only if it moves to a $P$-position. By
Proposition~\ref{prop:partition}, exactly one of the following alternatives
holds:
\begin{enumerate}
\item[(i)] $n=f(a,a)$ for some $a<n$;
\item[(ii)] $f(n,r)=n$ for some $r<n$.
\end{enumerate}
The value $a$ in (i) is unique: if $n=f(a,a)=f(b,b)$ with $a<b$, then $n>b$
by Lemma~\ref{lem:diagonalmax}, and the $P$-position $[n,a,a]$ is obtained
from $[n,b,b]$ by reducing the second row. This is impossible. The value $r$
in (ii) is unique by distinctness for fixed $n$, Lemma~\ref{lem:nonattacking}(a).

In case (i), the unique winning opening move is to $[n,a,a]$. In case (ii),
the unique winning opening move is to $[n,n,r]$. Therefore $[n,n,n]$ has
exactly one winning opening move.
\end{proof}

\begin{corollary}[The row of the winning opening move]\label{cor:openingrows}
With the row-numbering used in this note, no winning opening move from
$[n,n,n]$ is made in the first row. The move is made in the second row exactly
for the values
\[
        n\in D=\{f(a,a):a\ge 0\},
\]
and is then the move to $[n,a,a]$, where $a<n$ is uniquely determined. The
move is made in the third row exactly for the values
\[
        n\in S=\{p\ge 1: f(p,r)=p\text{ for some }r<p\},
\]
and is then the move to $[n,n,r]$, where $r<n$ is uniquely determined.
Equivalently, the second-row values are A029900 and the third-row values are
A029901.

The first terms of the two sequences of board sizes are
\[
\begin{array}{c|l}
\text{second row} & 1,3,4,6,8,10,11,13,15,16,18,20,21,24,25,27,\ldots\\
\text{third row}  & 2,5,7,9,12,14,17,19,22,23,26,29,31,33,36,38,\ldots
\end{array}
\]
\end{corollary}

\begin{proof}
A first-row move from $[n,n,n]$ leaves a smaller rectangle $[a,a,a]$. If
$a\ge 1$, this is an $N$-position by the standard strategy-stealing argument;
if $a=0$, the move takes the poisoned square. Hence no first-row opening move
is winning. The remaining two alternatives are exactly the alternatives in
Proposition~\ref{prop:partition} and in the proof of Theorem~\ref{thm:main}.
The OEIS identifications are the standard Chomp sequences listed under A029900
and A029901.
\end{proof}

\section*{Acknowledgements}
AI tools were used during the preparation of this note for exploratory work
with the recurrence, proof-structure brainstorming, and editorial revision of
drafts. The final mathematical arguments and responsibility for correctness
are the author's. As an independent check, the proof has been fully formalized
and machine-verified in the Lean 4 proof assistant, using only its standard
library.

\end{document}